\documentclass[11pt,bezier]{article}
\usepackage{amsmath,graphicx,amssymb,amsfonts}

\newtheorem{prethm}{{\bf Theorem}}

\newenvironment{thm}{\begin{prethm}{\hspace{-0.5
               em}{\bf}}}{\end{prethm}}
\newtheorem{prepro}{{\bf Theorem}}

\newtheorem{preprop}{{\bf Proposition}}

\newenvironment{prop}{\begin{preprop}{\hspace{-0.5
               em}{\bf}}}{\end{preprop}}
\newtheorem{preex}{{\bf Example}}

\newtheorem{precor}{{\bf Corollary}}

\newenvironment{cor}{\begin{precor}{\hspace{-0.5
               em}{\bf}}}{\end{precor}}
\newtheorem{preremark}{{\bf Remark}}

\newtheorem{prelem}{{\bf Lemma}}

\newenvironment{lem}{\begin{prelem}{\hspace{-0.5
               em}{\bf}}}{\end{prelem}}
\newtheorem{preproof}{{\bf Proof.}}

\newenvironment{proof}[1]{\begin{preproof}{\rm
               #1}\hfill{$\Box$}}{\end{preproof}}

\date{}

\begin{document}

\title{\bf On the signed graphs with two distinct eigenvalues}
\author{\large   F. Ramezani $^1$
  \\ {\it Department of Mathematics,}
\\ {\it  K.N.Toosi University of Technology, Tehran, Iran}
\\ {\it P.O. Box 16315-1618}
 \\[.3cm]
 {\it  School of Mathematics,}
\\ {\it Institute for Research in Fundamental Sciences (IPM),}
 \\{\it P.O. Box $19395$-$5746$, Tehran, Iran}}
\maketitle  \footnotetext[1]{\tt Email: ramezani@kntu.ac.ir.}

 {\bf Abstract}.
We consider signed graphs, i.e, graphs with positive or negative
signs on their edges. We construct some families of bipartite
signed graphs with only two distinct eigenvalues. This leads to
constructing infinite families of regular Ramanujan
graphs.

\textbf{Keywords:} Signed graphs, two distinct eigenvalues,
Ramanujan graphs.

\textbf{Mathematics Subject Classification:} 05C50,05C22.

\section{Introduction}
We consider only simple graphs, i.e., graphs with out loops and
multiple edges. The vertex set and edge set of the graph $G$ will
be denoted by $V(G)$ and $E(G)$, respectively. If there is no
doubt about $G$ we simply write $V$ and $E$. A \textit{signature}
on a graph $G$ is a function $s:E\rightarrow \{1,-1\}$. A graph
$G$ provided with a signature $s$ is called a \textit{signed
graph}, and will be denoted by $(G,s)$. We call the graph $G$ the
\textit{ground graph} of the signed graph $(G,s)$. The
\textit{adjacency matrix}, $A^s$ of the signed graph $(G,s)$ on
the vertex set $V=\{v_1,v_2,\ldots,v_n\}$, is an $n\times n$
matrix whose entries are $$A^s(i,j)=\left\{
\begin{array}{ll}
1, & \hbox{if $v_i$ is adjacent to $v_j$ and $s(\{v_i,v_j\})=1$;} \\
-1, & \hbox{ $v_i$ is adjacent to $v_j$ and $s(\{v_i,v_j\})=-1$;} \\
0, & \hbox{otherwise.}
\end{array}
\right.
$$ The $n\times n$ matrix $A$, whose entries
are the absolute values of the entries of $A^s$, is called the
\textit{ordinary adjacency matrix} of the graph $G$ and will be
denoted by $A(G)=|A^s|$. Note that any symmetric
$(0,\pm1)-$matrix $A$ with zero entries on diagonal is
corresponding to a signature of the graph whose adjacency matrix
is $|A|$. Then we some times use the symmetric matrix $A$ rather
than the signature which is in its correspondence. The
\textit{spectrum} of a signed graph is the eigenvalues of its
adjacency matrix. By $O_n$, $J_n$ and $I_n$ we mean the all zero
matrix, all one matrix, and identity matrix of size $n$,
respectively. If there is no doubt about size of the matrices we
simply write $O,J,I$. We say the $n\times n$ matrices $A$ and $B$
commute if $AB=BA$. An $n\times n$ matrix $A$ is called
anti-symmetric if $A^t=-A$, where $A^t$ is the transpose of the
matrix $A$. For an $n\times n$ signed matrix $A$, by $A^*$ we
mean the matrix $\begin{pmatrix}
  O_n & A \\
  A^t & O_n
\end{pmatrix}$. An $n\times n$
matrix $C$ is called \textit{orthogonal signed matrix} if the
entries of $C$ belongs to the set $\{0,1,-1\}$ and
$CC^t=C^tC=\alpha I_n$, where $\alpha$ is a positive integer. A
signed graph is called \textit{orthogonal} if its adjacency
matrix is an orthogonal signed matrix. Recently some problems on
the spectrum of signed adjacency matrices have attracted many
studies. In \cite{BE}, the authors have considered the lollipop
graph and proved its signed graphs are determined by their
spectrum. Energy of signed matrices has been considered in
\cite{EN}. The spectrum of signed graphs can be used to find the
spectrum of $2$-lifts of graphs, which leads to some results on
the existence of infinite families of regular Ramanujan graphs
with a fixed degree (cf. \cite{LPV}). It is known that the only
graphs with two distinct eigenvalues of the ordinary adjacency
matrix are the complete graphs. In this article we consider
bipartite graphs and find some signatures of them which lead to
signed graphs with only two distinct eigenvalues. A Ramanujan
graph, is a $d-$regular graph whose the second largest eigenvalue
is less than or equal to $2\sqrt{d-1}$. These graphs have several
applications to complexity theory, design of robust computer
networks, and the theory of error-correcting codes, see \cite{HO}.
In \cite{MSS} the authors have proved the existence of infinite
families of regular bipartite Ramanujan graphs of every degree
bigger than $2$. Here, using the Hadamard matrices, we construct
some regular bipartite Ramanujan graphs. A \textit{Conference
matrix} $Co_n$ is an $n\times n$ $(0,\pm1)-$matrix with $0$ on the
diagonal and $\pm1$ elsewhere such that $Co_n^tCo_n=(n-1)I_n$.
During the preparation of this paper we encounter interesting
results from other combinatorial areas such as regular two-graphs,
Hadamard matrices, Conference matrices and Orthogonal codes.
\section{preliminaries}
In this section we present some results which will be used to
construct signed graphs with only two distinct eigenvalues.

In a signed graph $(G,s)$ by \textit{resigning} at a vertex $v\in
V(G)$ we mean multiplying the signs of all the edges incident to
$v$ by $-1$. Two signed graphs $(G,s)$ and $(G,s')$ are called
\textit{equivalent} if one is obtained from the other by a
sequence of resigning around prescribed vertices. Otherwise we
call them \textit{distinct}. Most properties of signed graphs
specially the spectrum are the same in equivalent signed graphs.
In the following proposition from \cite{NEE} the number of
distinct signed graphs on the labeled graph $G$ is enumerated.
\begin{prop}  \label{dif}{\rm \cite{NEE} If the labeled graph $G$ has $m$ edges, $n$ vertices
and $c$ components, then there are $2^{(m-n+c)}$ distinct signed
graphs on it.}
\end{prop}
It is well-known that any $n\times n$ symmetric matrix has $n$
real eigenvalues. If
$\lambda_1^{m_1},\lambda_2^{m_2},\ldots,\lambda_k^{m_k}$ are the
eigenvalues of a symmetric matrix considering their multiplicity
then the minimal polynomial of it will be equal to
$M(A,x)=\prod_{i=1}^k (x-\lambda_i)$. Hence any symmetric matrix
with only two distinct eigenvalues has a minimal polynomial of
the following form.
$$x^2+ax+b=0.$$

In the following lemmas we mention the basic properties of
possible signed graphs with just two distinct eigenvalues.
\begin{lem} {\rm Suppose the signed
graph $(G,s)$ has just two distinct eigenvalues, then the graph
$G$ is regular. }
\end{lem}
\begin{proof} {The diagonal entries of the matrix $(A^s)^2$ are the vertex
degrees of $G$. On the other hand if $\lambda,\mu$ are the
distinct eigenvalues of $(G,s)$, then $(A^s)^2-(\lambda+\mu)
A^s+\lambda\mu I=O$. The diagonal entries of the matrix $A^s$ are
zero hence the diagonal entries of $(A^s)^2$ (therefore the vertex
degrees of $G$) should be equal to $-\lambda\mu$, this implies
that $G$ is regular.}
\end{proof}
\subsection{Complete signed graphs}
In this short section we review the existing results on the
signed graph $(K_n,s)$ with only two distinct eigenvalues, where
$K_n$ is the complete graph with $n$ vertices.

A two-graph $\Gamma$ is a set consisting of $3-$subsets of a
finite set $X$, say set of vertices, such that every $4-$subset
of $X$ contains an even number of $3-$subsets of $\Gamma$. A
two-graph is called regular if every pair of vertices lies in the
same number of $3-$subsets of the two-graph. Let $\Gamma$ be a
two-graph on the set $X$. For any $x \in X$, we define a graph
with vertex set $X$ where vertices $y$ and $z$ are adjacent if and
only if $\{x, y, z\}$ is in $\Gamma$. 
For a graph $G$, a signed complete graph $\Sigma$ has been
corresponded on the same vertex set, whose edges have negative
sign if it is an edge of $G$ and positive sign otherwise. Actually
the graph $G$ consists of all vertices and all negative edges of
$\Sigma$. The signed adjacency matrix of a two-graph is the
adjacency matrix of the corresponding signed complete graph.

 In
\cite{TA} it is proved that a two-graph is regular if and only if
its signed adjacency matrix has just two distinct eigenvalues.
Hence any regular two graph gives a complete signed graph with
only two distinct eigenvalues.
For more details about two-graphs see \cite{BCN}. 
%

\textbf{Example. } The following set of $3-$subsets is a regular
two-graph on the set of vertices $\{1,2,\ldots,6\}$.
    $$\{1,2,3\},\{1,2,4\},\{1,3,5\},\{1,4,6\},\{1,5,6\},\{2,3,6\},\{2,4,5\},\{2,5,6\},\{3,4,5\},\{3,4,6\}.$$
For the vertex $1$ of the two-graph, the corresponding graph $G$
is on the vertex set $\{1,\ldots,6\}$ and the edges are
$23,24,35,46,56.$ Hence the corresponding signed complete graph
have the following adjacency matrix. $$A=\begin{pmatrix}
  0 & 1 & 1 & 1 & 1 & 1 \\
  1 & 0 & -1 & -1 & 1 & 1 \\
  1 & -1 & 0 & 1 & -1 & 1 \\
  1 & -1 &1 & 0 & 1 & -1 \\
  1 & 1 & -1 & 1 & 0 & -1 \\
  1 & 1 & 1 & -1 & -1 & 0
\end{pmatrix}.$$
The spectrum of $A$ is $ [2.2361^3,- 2.2361^3]$.

\subsection{Bipartite signed graphs}
In this section we focus on the signed graphs with only two
distinct eigenvalues where the ground graph is bipartite.
\begin{lem} {\rm Let $G$ be a bipartite graph on $2n$ vertices, then the signed graph $(G,s)$ has only two distinct
eigenvalues if and only if its adjacency matrix is of the form
$\begin{pmatrix}
    O_n & C \\
    C^t & O_n\
  \end{pmatrix},$ where $C$ is an orthogonal signed matrix of size $n$.
}
\end{lem}
\begin{proof} {\rm If $G$ is a bipartite graph on $2n$ vertices, then the adjacency
matrix of $(G,s)$ is of the block form $\begin{pmatrix}
    O_n & C \\
    C^t & O_n\
  \end{pmatrix},$ on the other hand we have $(A^s)^2+\alpha
A^s+\beta I_{2n}=O_{2n}$, for some integers $\alpha,\beta$. It
implies
$$\begin{pmatrix}
    O_n & C \\
    C^t & O_n\
  \end{pmatrix}^2+\alpha\begin{pmatrix}
    O_n & C \\
    C^t & O_n\
  \end{pmatrix}+\beta I_{2n}=O_{2n}.$$ Therefore, $$\begin{pmatrix}
    CC^t & O_n \\
    O_n & C^tC\
  \end{pmatrix}+\alpha\begin{pmatrix}
    O_n & C \\
    C^t & O_n\
  \end{pmatrix}+\beta I_{2n}=O_{2n}.$$The above equality yields $$CC^t=C^tC=-\beta I_{n},
   \textrm{  } \alpha C=O_n.$$ Hence $\alpha=0$ and $C$ and $A^s$ are orthogonal signed matrices.
   On the other hand suppose that $C$ is an orthogonal signed matrix of size $n$, with $CC^t=C^tC=\gamma I$.
   Then for the symmetric signed matrix $A^s=C^*,$ we have $(A^s)^2=\gamma I_{2n}$ and hence
    the corresponding signed graph has just two distinct eigenvalues $\pm\sqrt{\gamma}$.   }
\end{proof}
The above lemma implies that any bipartite signed graph with just
two distinct eigenvalues is in correspondence with an orthogonal
signed matrix $C$. Therefore from now on we will refer to
orthogonal signed matrices instead of bipartite signed graphs
with just two distinct eigenvalues. We will denote the
corresponding orthogonal signed matrix of a bipartite signed
graph $(G,s)$ by $C_s(G).$
\begin{cor} {\rm If $G$ is a complete bipartite graph with $2n$ vertices, and $(G,s)$ has only two distinct
eigenvalues, then the matrix $C_s(G)$ is an Hadamard matrix of
size $n$.}
\end{cor}
\begin{cor} {\rm If $G$ is a bipartite $(n-1)$-regular graph on $2n$ vertices, and $(G,s)$
has only two distinct eigenvalues, then the matrix $C_s(G)$ is a
Conference matrix of size $n$.}
\end{cor}

\section{Constructing bipartite signed graphs with two distinct eigenvalues}
In this section we first recall some well known methods of
constructing the Hadamard matrices, then apply them to find
orthogonal signed matrices and hence bipartite signed graphs with
just two distinct eigenvalues.

\textbf{Definition.} If $A$ is an $m\times n$ matrix and $B$ is a
$p \times q$ matrix, then the Kronecker product $A\otimes B$ is
the $mp \times nq$ block matrix, $$A\otimes B=\begin{pmatrix}
  a_{11}B & \cdots & a_{1n}B \\
  \vdots & \ddots & \vdots \\
  a_{m1}B & \cdots & a_{mn}B
\end{pmatrix},$$

It is proved that if the spectrum of an $n\times n$ matrix $A$
and an $m\times m$ matrix $B$ are
$\lambda_1,\lambda_2,\ldots,\lambda_n$ and
$\mu_1,\mu_2,\ldots,\mu_m$, respectively, then the spectrum of
$A\otimes B$ is $\lambda_i\mu_j,$ for $i=1,2,\ldots,n$ and
$j=1,2,\ldots,m$, cf. \cite{LIWI}.

The following lemma from \cite{LIWI} will be used for
constructing orthogonal signed matrices from small examples.
\begin{lem} {\rm Let $A$ and $B$ are orthogonal matrices, then the matrix
$A\otimes B$, which is the Kronecker product of $A,B$, is
orthogonal.}
\end{lem}
The following corollaries are immediate consequences of the above
lemma.
\begin{cor} {\rm If there exist bipartite signed graphs on respectively $2n,2m$
vertices, with only two distinct eigenvalues, then there exists a
bipartite signed graph on $4mn$ vertices, which has only two
distinct eigenvalues.}
\end{cor}
\begin{cor} {\rm For any positive integer $n$, there is an orthogonal
signed matrix of size $2n$, and hence a bipartite signed graph on
$4n$ vertices and only two distinct eigenvalues }
\end{cor}
\begin{proof} {Consider the matrix $A=I_n\otimes H_2$,
where $H_2$ is the Hadamard matrix of size $2$, i.e
$$H_2=\begin{pmatrix}
  1 & 1 \\
  1 & -1
\end{pmatrix}.$$
The matrix $A$ is an orthogonal matrix with spectrum $\pm\sqrt{2}^n$, hence the
symmetric matrix $A^*$ will be the adjacency matrix of a signed
bipartite graph with only two distinct eigenvalues. }
\end{proof}
\begin{lem} {\rm Let $C$ be a symmetric $n\times n$ orthogonal signed matrix with zero
entries on the diagonal. Then the matrix defined bellow is a
symmetric orthogonal signed matrix,
$$B=\begin{pmatrix}
  C+I & C-I \\
  C-I & -C-I
\end{pmatrix}.$$}
\end{lem}
\begin{proof} {Since $C$ is symmetric so is $B$. Suppose that $C^2=\alpha I,$ we have $$ {\small B^2=\begin{pmatrix}
  (C+I)^2+(C-I)^2 & (C+I)(C-I)+(C-I)(-C-I) \\
  (C-I)(C+I)+(-C-I)(C-I) & (C-I)^2+(-C-I)^2
\end{pmatrix},}$$$${\small =\begin{pmatrix}
  C^2+2I & O \\
  O & C^2+2I
\end{pmatrix}=\begin{pmatrix}
  (\alpha+2)I & O \\
  O & (\alpha+2)I
\end{pmatrix},}
$$
so the assertion follows.}
\end{proof}
\begin{lem} {\rm Let $C$ be an antisymmetric orthogonal matrix with zero entries on the diagonal.
Then the matrix $C+I$ is an orthogonal matrix. }
\end{lem}
\begin{proof} {Suppose that $CC^t=\alpha I$ for some positive integer $\alpha$.
We have $$(C+I)(C^t+I)=CC^t+C+C^t+I=(\alpha+1) I,$$ since $C$ is an antisymmetric
orthogonal signed matrix. Hence the assertion follows.}
\end{proof}

In the following theorem, we establish a method for constructing
orthogonal matrices, based on the Williamson method (cf.
\cite{LIWI}).
\begin{thm}\label{Wil}{ \rm Let the matrices $A_i, 1\leq i\leq 4$, be symmetric of order
$n,$ with entries $0,\pm1$, having constant number of zeros in
each row and assume that they are mutually commuting with each
other. Consider the matrix $H$ defined bellow, $$H=\left(
                                     \begin{array}{cccc}
                                       A_1 & A_2 & A_3 & A_4 \\
                                       -A_2 & A_1 & -A_4 & A_3 \\
                                       -A_3 & A_4 & A_1 & -A_2 \\
                                       -A_4 & -A_3 & A_2 & A_1 \\
                                     \end{array}
                                   \right),$$ then $H$ is an orthogonal
                                   matrix if and only if $$\sum_{i=1}^4 A_i^2=(\sum_{i=1}^4 k_i)I.$$
                                   Where $k_i$ is the number of non-zero entries in each row of $A_i,$ for $i=1,\ldots,4.$
}
\end{thm}
\begin{proof} {The assertion follows by simply calculating $HH^t$.}
\end{proof}

\begin{prop} {\rm Let $C$ be a symmetric orthogonal signed matrix of order $n$. Then the matrix
$H$, in the previous theorem, obtained by substituting any of the
following matrices, is an orthogonal signed matrix.

\begin{itemize}
  \item $A_1=A_2=A_3=A_4=C$, or
  \item $A_1=A_2=C$, $A_3=C-I$ and $A_4=C+I$, or
  \item $A_1=A_2=C+I$, $A_3=A_4=C-I$,
\end{itemize}
providing the matrix
  $C$ has zero diagonal in the second and last cases.

}
\end{prop}
\begin{prop} {\rm If the matrices $A_i$ in Theorem \ref{Wil} are equal to a
given orthogonal matrix $C$, then the matrix $H$ will be
orthogonal.}
\end{prop}
\begin{proof} {In general the matrices $A_i$, defined in Theorem \ref{Wil} do not need to be symmetric.
By calculation we see that the only requirement for the
orthogonality of $H$ is that the matrices $A_iA^t_j$ for
$i,j=1,\ldots,4$ need to be symmetric. Since we have $CC^t=\alpha
I,$ then the assertion follows. }
\end{proof}
One may consider the matrix $C$ in the previous propositions to
be a symmetric conference matrix. An infinite family of symmetric
conference matrices exist, for instance with Paley method we may
construct symmetric conference matrices of sizes $q+1$, where $q$
is a prime power and $q\equiv 1\textrm{ }(\textrm{mod } 4)$. (cf.
\cite{LIWI}, pp. 175--176)

\section{Some Ramanujan graphs}
Bilu and Linial conjectured that every regular Ramanujan graph of
degree $d$ has a signature $s$ on the edge set with the property
$\lambda_1(A^s)\leq 2\sqrt{d-1}$, see \cite{BL}. We will call the
mentioned signature a good signature. In this section for some
families of known Ramanujan graphs, we verify the conjecture of
Bilu and Linial and introduce new families of Ramanujan graphs.
To state our main result we need some lemmas.

It is well-known that if $n\times n$ symmetric matrices $A$ and
$B$ commute then the eigenvalues of $A+B$ are the sum of
eigenvalues of $A$ and $B$, i.e
$$\lambda_i(A+B)=\lambda_i(A)+\lambda_i(B),\textrm{ } i=1,2,\ldots,n.$$

\begin{lem} \label{ram}{\rm If $G$ is a $k-$regular graph on $n$ vertices and
$\frac{1}{4}(k-1)^2+k+2\leq n$ then the complement of $G$, $G^c$
is a Ramanujan graph.}
\end{lem}
\begin{proof} {The graph $G^c$ is a $(n-k-1)-$regular graph, hence
we need to prove that $$\lambda_2(G)\leq 2\sqrt{n-k-2}.$$ It is
well-known that the eigenvalues of the complement of $G$ can be
given in decreasing order as in the following list,
$$n-k-1,-1-\lambda_n(G),-1-\lambda_{n-1}(G),\ldots,-1-\lambda_2(G).$$
By Perron-Frobenius Theorem we know that $\lambda_n(G)\geq -k$
then $\lambda_2(G^c)\leq k-1,$ and the assertion follows by the
following chain of inequalities,
$$\lambda_2(G^c)\leq k-1\Leftrightarrow \frac{1}{4}\lambda^2_2(G^c)\leq \frac{1}{4}(k-1)^2\leq n-(k+2),$$
hence $$\lambda_2(G^c)\leq 2\sqrt{n-k-2}$$ thus $G^c$ is a
Ramanujan graph.    }
\end{proof}
\begin{prop} {\rm Let $C$ be a symmetric orthogonal signed matrix of order
$n$, with $CC^t= \alpha I,$ $\alpha\geq 2$ and $k=n-1-\alpha,$
holds in the inequality $\frac{1}{4}(k-1)^2+k+2\leq n$ then the
ground graph $G$ corresponding to the matrix $C$ is a Ramanujan
graph and $C$ is a good signature of $G$. }
\end{prop}
\begin{proof} {The graph $G$ is $\alpha-$regular,
on the other hand $C$ is a symmetric orthogonal signed matrix
hence $\lambda_1(C)=\sqrt{\alpha}\leq 2\sqrt{\alpha-1}$. The
complement of $G$ is a regular graph of degree $k=n-1-\alpha.$ By
the assumption and Lemma \ref{ram} the assertion follows. }
\end{proof}
In the case that the orthogonal signed matrix $C$ is not
symmetric we can use it to construct bipartite Ramanujan graphs.
But before that we need some preliminaries. For a bipartite graph
$G$ with bipartition $X,Y$ of vertices, the \textit{bipartite
complement} of $G$ which is denoted by $G_b^c$ is a bipartite
graph on the vertex set $V=X\cup Y$, where a vertex $x\in X$ is
adjacent to a vertex $y\in Y$ if and only if $x$ was not adjacent
to $y$ in $G$.

\begin{lem} {\rm If $G$ is a $k-$regular bipartite Ramanijan graph on the disjoint
sets $X,Y$ of size $n$, and $k\leq \frac{n}{2}$ then the
bipartite complement of $G$ is also a Ramanujan graph.  }
\end{lem}
\begin{proof} {We have $A(G)+A(G_b^c)=A(K_{n,n})$. The graph $G$ is regular and bipartite, hence the
adjacency matrix of $G$ is of the following form.
$$A(G)=\begin{pmatrix}
  O & B \\
  B^t & O
\end{pmatrix},$$where $B$ is a $(0,1)$-matrix with a fix number of $1$'s on each row and column.
Therefore the matrices $A(G)$ and $A(K_{n,n})$ commute and hence
the eigenvalues of $G_b^c$ follows
$$\lambda_i(G_b^c)=\lambda_i(K_n,n)-\lambda_i(G)=
  \begin{cases}
    n-k & \text{$i=1$}, \\
    \lambda_i(G) & \text{$1<i<2n$,}\\
    k-n& \text{$i=2n$}.
  \end{cases}
$$ By the assumption $\lambda_2(G)\leq 2\sqrt{k-1}$, and $k\leq n-k$ hence
$$\lambda_2(G_b^c)=\lambda_2(G)\leq 2\sqrt{k-1}\leq2\sqrt{n-k-1},$$ therefore $G_b^c$ is a Ramanujan graph. }
\end{proof}

\begin{lem} {\rm The bipartite complement of the graph $kC_4$, ($k\geq 2$) that is the $k$
disjoint copy of $C_4$, is a Ramanujan graph. }
\end{lem}
\begin{proof}{As disscused in the proof of the above lemma, the eigenvalues of $(kC_4)^c_b$
can be calculated by the eigenvalues of the graph $K_{2k,2k}$ and
the graph $kC_4$. Thus the spectrum of $(kC_4)^c_b$ is
$\pm(2k-2),(\pm 2)^{k-1},0^{2k}$, hence the assertion follows.  }
\end{proof}

\begin{lem} {\rm If $C$ is an orthogonal signed matrix of size $n$ and the ground graph $G$
corresponding to  $C^*$ is a Ramanujan graph, then $C^*$ is a
good signature of $G$. }
\end{lem}

For a signature $s$ on the edges of the graph $G$ the $2-$lift
$\tilde{G}_s$ of the graph $G$ is a graph on the vertex set
$V(G)\times \{1,2\}$, where there is an edge between $(u,i)$ and
$(v,j)$ if and only if
$$ \{u,v\}\in
\textrm{E}(G) \textrm{ and }
  \begin{cases}
    i=j& \text{ if }s(uv)=1, \\
    i=\bar{j} & \text{otherwise}.
  \end{cases}
,$$where $\bar{1}=2$ and $\bar{2}=1.$  In Figure 1, we illustrate
the definition. Note that in the figure the negative edge is
denoted by dot line. In the graph $\tilde{G}_s$, the vertices of
$(G,s)$ are duplicated and the positive edges are replaced by two
parallel lines and the negative edge is replaced by two crossed
edges.

\vspace{1cm}
${\hspace{1cm}\put(100,0){\line(0,1){50}}\put(100,0){\line(1,1){50}}\put(150,0){\vdots
 }\put(150,12){\vdots}\put(150,24){\vdots}\put(150,36){\vdots}\put(150,48){.}
 \put(150,0){\line(-1,1){50}}\put(100,0){\circle*{5}}\put(150,50){\circle*{5}}
 \put(100,50){\circle*{5}}\put(150,0){\circle*{5}}\put(170,25){$\longrightarrow$}
 \put(210,0){\circle*{4}}\put(220,0){\circle*{4}}\put(210,50){\circle*{4}}\put(220,50){\circle*{4}}
 \put(270,0){\circle*{4}}\put(260,0){\circle*{4}}\put(270,50){\circle*{4}}\put(260,50){\circle*{4}}
 \put(210,0){\line(0,1){50}}\put(220,0){\line(0,1){50}}\put(210,0){\line(1,1){50}}\put(220,0){\line(1,1){50}}
 \put(270,0){\line(-1,1){50}}\put(260,0){\line(-1,1){50}}\put(270,0){\line(-1,5){10}}\put(260,0){\line(1,5){10}}}$
$$\hspace{-1cm}(G,s)\hspace{3.5cm}\tilde{G}_{s}$$
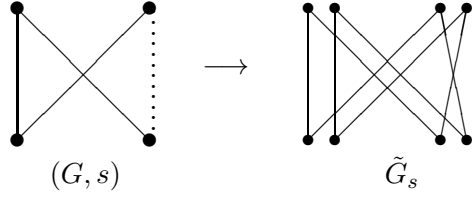
 $$\textrm{\textbf{Figure1.} 2-\textrm{lift of } G \textrm{ corresponding to the specific signature}}$$


 The spectrum of the graph $\tilde{G}_s$ is determined in
\cite{LPV}.
\begin{lem} \label{eig1} {\rm \cite{LPV} The spectrum of $A(\tilde{G}_s)$ is the
multi set union of spectrum of $A$ and $A^s.$}
\end{lem}

By the above Lemma the authors of the paper \cite{MSS} conclude
that for any Ramanujan graph $G$ and a good signature $s$ of it,
the graph $\tilde{G}_s$ is a Ramanujan graph. As a final result,
in the following table we list some Ramanujan graphs and a good
signature of them. Suppose that $H_n$ is a Hadamard matrix and
$Co_n$ is a conference matrix of order $n$.

\vspace{1cm} {\hspace{2cm} \begin{tabular}{|c|c|c|}\hline
  $G$ & a good signature $A^s$ & eigenvalues of $\tilde{G}_s$ \\\hline
  $K_{n,n}$ & $H_n^*$& $\pm n,\pm\sqrt{n}^{n},0^{2n-2}$ \\\hline
  $K_{n,n}\setminus M$ & $Co_n^*$ & $\pm (n-1),\pm\sqrt{n-1}^{n},\pm 1^{n-1}$ \\\hline
 $(nC_4)^c_b$ &$\begin{pmatrix}
   Co_n & Co_n \\
   -Co_n & Co_n \
 \end{pmatrix}^*$ & $\pm(2n-2),(\pm 2)^{n-1},0^{2n}$ \\\hline
 \end{tabular}}

\paragraph{Acknowledgment} We would like to thank B. Tayfeh-Rezaie
for attracting our attention to the problem studied in this paper.
We also would like to thank the reviewer for his/her detailed
comments and useful suggestions. This research was in part
supported by a grant from IPM (No. 95050013).

\end{document}